\documentclass[reqno]{amsart}
\title[On series identities of Gosper and integrals...]
{On series identities of Gosper and integrals of Ramanujan theta function $\psi(q)$}
\usepackage{amssymb,amsmath,amsthm,epsfig,graphics,latexsym}
\usepackage{enumerate}
\usepackage{color,soul}
 \theoremstyle{definition}

  \theoremstyle{plain}

  \newtheorem{theorem}    {Theorem}
  
  \newtheorem{corollary}  {Corollary}

  \theoremstyle{definition}
  \newtheorem{remark}{Remark}

  \newcommand{\fr}{\frac}

\begin{document}
  \author{Mohamed El Bachraoui}
  \address{Department of  Mathematical Sciences,
 United Arab Emirates University, PO Box 15551, Al-Ain, UAE}
 \email{melbachraoui@uaeu.ac.ae}
\begin{abstract}
We prove some Lambert series which were stated by Gosper without proof or reference.
As an application we shall evaluate integrals involving Ramanujan theta function $\psi(q)$.
Furthermore, motivated by Ramanujan's identities for $q\psi^4(q^2)$ and $\fr{\psi^3(q)}{\psi(q^3)}$,
we shall evaluate the squares of $q\psi^4(q^2)$ and $\fr{\psi^3(q)}{\psi(q^3)}$ in terms of Lambert series.
\end{abstract}
\date{\textit{\today}}
\maketitle
\section{Introduction}
\noindent
Throughout the paper let $q=e^{\pi i\tau}$ with $\mathrm{Im}(\tau)>0$ guaranteeing
that $|q|<1$.
As usual, the $q$-shifted factorials for a complex number $a$ are given by
\[
(a;q)_0= 1,\quad (a;q)_n = \prod_{i=0}^{n-1}(1-a q^i),\quad
(a;q)_{\infty} = \lim_{n\to\infty}(a;q)_n.
\]
Gosper~\cite{Gosper} introduced the following $q$-analogues of $\sin(z)$ and $\cos(z)$
\begin{equation}\label{sine-cos-q-prod}
\begin{split}
\sin_q (\pi z) &= 
q^{(z-\frac{1}{2})^2} \frac{(q^{2z};q^2)_{\infty} (q^{2-2z};q^2)_{\infty}}{(q;q^2)_{\infty}^2} \\
\cos_q (\pi z) 
&= q^{z^2} \frac{(q^{1+2z};q^2)_{\infty} (q^{1-2z};q^2)_{\infty}}{(q;q^2)_{\infty}^2}
\end{split}
\end{equation}
along with the following $q$-constant
\begin{equation}\label{Piq}
\Pi_q = q^{\frac{1}{4}}\frac{ (q^2;q^2)_{\infty}^2}{(q;q^2)_{\infty}^2} =
q^{\frac{1}{4}} \psi^2(q)
\end{equation}
for which one has $\lim_{q\to 1} (1-q^2)\Pi_q = \pi$.
The Ramanujan theta functions $f(q)$, $\varphi(q)$, and $\psi(q)$ are given by
\[
f(-q)= (q;q)_{\infty},\quad \varphi(q)= \fr{(-q;-q)_{\infty}}{(q;-q)_{\infty}},\quad
\psi(q) = \fr{(q^2;q^2)_{\infty}}{(q;q^2)_{\infty}} 
\]
but our focus in this work is mainly on $\psi(q)$.
We refer the reader to
Berndt~\cite{Berndt-1} and Cooper~\cite{Cooper} for a survey on Ramanujan theta functions, the connections between them, and their applications.
Our primary goal in this paper is to confirm the following formulas which were stated without proof
or reference by Gosper~\cite[p. 102]{Gosper} and to give some of their applications.
\begin{theorem}\label{Gosper-sums} We have
\[
\begin{split}
{\mathrm(a)\qquad} &
\sum_{n=1}^{\infty}\frac{q^n}{(1-q^n)^2} - 2 \sum_{n=1}^{\infty}\frac{q^{2n}}{(1-q^{2n})^2} =
\frac{1}{24}\left( \frac{\Pi_q^4}{\Pi_{q^2}^2}-1 \right) + \frac{2}{3}\Pi_{q^2}^2. \\
{\mathrm(b)\qquad} &
\sum_{n=1}^{\infty}\frac{q^n}{(1-q^n)^2} - 3 \sum_{n=1}^{\infty}\frac{q^{3n}}{(1-q^{3n})^2} =
\fr{(\Pi_q^2 + 3\Pi_{q^3}^2)^2}{12\Pi_q \Pi_{q^3}} - \fr{1}{12}. \\
{\mathrm(c)\qquad} &
\sum_{n=1}^{\infty}\frac{q^n}{(1-q^n)^2} - 4 \sum_{n=1}^{\infty}\frac{q^{4n}}{(1-q^{4n})^2} =
\fr{1}{8}\left( \fr{\Pi_q^4}{\Pi_{q^2}^2}-1\right). \\
{\mathrm(d)\qquad} &
\sum_{n=1}^{\infty}\frac{q^{2n}}{(1-q^{2n})^2} - 9 \sum_{n=1}^{\infty}\frac{q^{18n}}{(1-q^{18n})^2} =
\frac{\Pi_{q^3}^3}{\Pi_{q}} + \frac{1}{3}\left( \frac{\Pi_{q^3}^3}{\Pi_{q^9}}-1 \right). \\
{\mathrm(e)\qquad} &
\sum_{n=1}^{\infty}\frac{q^{2n-1}}{(1-q^{2n-1})^2} - 3 \sum_{n=1}^{\infty}\frac{q^{6n-3}}{(1-q^{6n-3})^2} =
\Pi_{q^3}\Pi_{q}. \\
{\mathrm(f)\qquad} &
6 \sum_{n=1}^{\infty}\frac{q^{4n-2}}{(1-q^{2n-1})^4} + \sum_{n=1}^{\infty}\frac{q^{2n-1}}{(1-q^{2n-1})^2} =
\Pi_q^4.
\end{split}
\]
\end{theorem}
To prove some of our identities we shall also need the following result which apparently is new.
\begin{theorem}\label{Beyond-Gosper}
We have
\[
\sum_{n=1}^{\infty}\Big( \fr{q^{2n}}{(1-q^{2n})^2}-\fr{q^{6n}}{(1-q^{6n})^2}-
\fr{2 q^{6n-3}}{(1-q^{6n-3})^2} \Big)
= \fr{\Pi_{q^3}^3}{\Pi_q}.
\]
\end{theorem}
\begin{remark}\label{rmk:divisor}
Recall for the divisor function $\sigma(n)=\sum_{d\mid n}d$ the basic fact that
\[
\sum_{n=1}^{\infty}\sigma(n) q^n = \sum_{n=1}^{\infty}\fr{q^n}{(1-q^n)^2} = \sum_{n=1}^{\infty}\fr{nq^n}{1-q^n}.
\]
Now letting 
\[
\sigma^{\ast}_{b(a)}(n) = \sum_{\substack{d\mid n\\ \fr{n}{d}\equiv b\bmod a}} d,
\]
we clearly have $\sigma^{\ast}_{0(1)} = \sigma(n)$ and $\sigma^{\ast}_{1(2)}(2n+1)= \sigma(2n+1)$. More importantly, we can verify that
\begin{equation}\label{divisor-generating}
\sum_{n=1}^{\infty} \sigma^{\ast}_{b(a)}(n) q^n
= \begin{cases}
\sum_{n=0}^{\infty} \fr{q^{an+b}}{(1-q^{an+b})^2} & \text{if\ } b\not=0 \\
\sum_{n=1}^{\infty} \fr{q^{an}}{(1-q^{an})^2} & \text{if\ } b=0.
\end{cases}
\end{equation}
Therefore, with the exception of part~(f), each one of the series in Theorems~\ref{Gosper-sums}~and~\ref{Beyond-Gosper} can be written in terms
of $\sigma^{\ast}_{b(a)}(n)$.
\end{remark}
\begin{remark}
Note that Theorem~\ref{Gosper-sums}(e) is essentially equivalent to the Ramanujan's identity
\[
q\psi^2(q)\psi^2(q^3) = \sum_{n=1}^{\infty}\fr{nq^n}{1-q^{2n}} - 3\sum_{n=1}^{\infty}\fr{n q^{3n}}{1-q^{6n}},
\]
see Berndt~\cite[p. 223]{Berndt-1}. We include it here as our proof is different.
\end{remark}
Our primary application is motivated by Ramanujan's integral identities of the theta functions which are found in~\cite[p. 207]{Ramanujan}. 
Ramanujan recorded six
integrals involving the functions $f(q)$, $\varphi(q)$, and $\psi(q)$. Son~\cite{Son} gave complete proofs for five of these integrals and a partial proof for the sixth
integral stating that
\begin{equation}\label{Ram-integral-1}
q^{\fr{1}{9}} \prod_{n=0}^{\infty} \fr{(1-q^{3n+1})^{3n+1}}{(1-q^{3n+2})^{3n+2}}
= \exp\left(-C - \fr{1}{9}\int_q^1 \fr{f^9(-t)}{f^3(-t^3)}\fr{dt}{t} \right).
\end{equation}
Berndt~and~Zaharescu~\cite{Berndt-Zaharescu} found a complete proof for (\ref{Ram-integral-1}) and established Ramanujan's constant $C$. Extensions for the
integral identity~(\ref{Ram-integral-1}) along with a method to compute related constants $C$ are obtained in Ahlgren~\emph{et al.}~\cite{Ahlgren-et-al}.
Besides, among the six integral formulas of Ramanujan we find
\begin{equation}\label{Ram-integral-2}
\exp\left( -2 \int_0^q \psi^4(t^2)\, dt\right) = \fr{\psi(-q)}{\psi(q)} \ \text{and\ }
\exp\left( 2 \int_0^q \psi^2(t)\psi^2(t^3)\, dt\right) = \fr{\varphi(-q)}{\varphi(q)}.
\end{equation}
Observe that if we consider indefinite integrals rather than definite integrals in the two identities
of~(\ref{Ram-integral-2}), then we get Lambert series involved as values.
Indeed, as to the first formula in~(\ref{Ram-integral-2}),
on taking logarithms and expanding as powers series, we get
\begin{equation}\label{equiv-int-1}
\begin{split}
\int \psi^4(q^2)\, dq &= \fr{-1}{2} \log\fr{\psi(-q)}{\psi(q)} \\
&= -\fr{1}{2}\sum_{n=0}^{\infty} \big(\log(1-q^{2n+1}) - \log(1+q^{2n+1})\big) \\
&= \sum_{n=1}^{\infty}\sum_{m=1}^{\infty} \fr{q^{(2n-1)(2m-1)}}{2m-1} \\
&= \sum_{n=1}^{\infty} \fr{1}{2n-1}\fr{q^{2n-1}}{1-q^{4n-2}}.
\end{split}
\end{equation}
Likewise the second formula in~(\ref{Ram-integral-2}) is equivalent to
\begin{equation}\label{equiv-int-2}
\begin{split}
\int \psi^2(q)\psi^2(q^3)\, dq
&= \sum_{n=1}^{\infty} \fr{1}{2n-1}\left(\fr{q^{2n-1}}{1-q^{2n-1}} - \fr{q^{6n-3}}{1-q^{6n-3}} \right) \\
&= \sum_{n=1}^{\infty}\fr{q^{2n-1}+q^{4n-2}}{(2n-1)(1-q^{6n-3})}.
\end{split}
\end{equation}
%
It is also worth to notice that  $\int \psi^8(q) \, dq$ can be obtained by using Ramanujan's famous
formula for $\psi^8(q)$ (see Berndt~\cite[p. 138]{Berndt}) as follows
\[
\begin{split}
\int \psi^8(q) \, dq
&= \int \sum_{n=1}^{\infty}\fr{n^3 q^{n-1}}{1-q^{2n}} \,dq \\
&= C+ \sum_{n=1}^{\infty}\fr{n^2}{2}\Big(\log(1+q^n)-\log(1-q^n)\Big) \\
&= C+ \sum_{n=1}^{\infty}\sum_{m=1}^{\infty}\fr{n^2 q^{n(2m-1)}}{2m-1}.
\end{split}
\]

\noindent
We have in this direction the following results including a different formula for $\int \psi^8(q) \, dq$.
\begin{theorem}\label{main-integrals}
We have
\[
\begin{split}
{\mathrm(a)\qquad}
\int \psi^8(q) \, dq
&= C + \sum_{n=1}^{\infty}\fr{1}{2n-1}\fr{q^{2n-1}+q^{4n-2}}{(1-q^{2n-1})^3} \\
{\mathrm(b)\quad}
\int \fr{\psi^8(q)}{q\psi^4(q^2)} \, dq
&= C + \log q + 8\sum_{n=1}^{\infty}\fr{1}{n}\fr{q^n + q^{2n} + q^{3n}}{1-q^{4n}} \\
&= C + \log q + \sum_{n=1}^{\infty}\left(\fr{24 q^n}{n(1-q^{2n})} - \fr{16 q^{2n-1}}{(2n-1)(1-q^{4n-2})}\right) \\
{\mathrm(c)\quad}
\int \fr{\psi^8(q^2)}{q\psi^4(q^4)} \, dq
&= C + \log q + \sum_{n=1}^{\infty}\fr{4(q^{2n} + q^{4n} + q^{6n})}{n(1-q^{8n})} \\
{\mathrm(d)\quad}
\int \fr{q \psi^6(q^3)}{\psi^2(q)} \, dq
&= C + \sum_{n=1}^{\infty}\left(\fr{2 + 3q^{2n} + 3q^{4n}}{6n(1-q^{6n})}
- \fr{2}{(6n-3)(1-q^{6n-3})}\right) \\
{\mathrm(e)\quad}
\int \fr{\psi^6(q^3)}{q \psi^2(q^9)} \, dq
&= C + \log q+ \sum_{n=1}^{\infty}\left(\fr{-2 + 3q^{6n} + q^{12n}}{2n(1-q^{18n})}
+ \fr{2}{(2n-1)(1-q^{6n-3})} \right) \\
{\mathrm(f)\quad}
\int \fr{\psi^6(q)}{q \psi^2(q^3)} \, dq
&= C + \log q- \sum_{n=1}^{\infty}\left(\fr{6 + 9q^{2n} + 9q^{4n}}{2n(1-q^{6n})}
+ \fr{3-12q^{2n-1}-12q^{4n-2}}{(2n-1)(1-q^{6n-3})} \right). \\
\end{split}
\]
\end{theorem}

\noindent
Our secondary application is related to the following Lambert series representations for $q \psi^4(q^2)$
and $\fr{\psi^3(q)}{\psi(q^3)}$ which are due to Ramanujan's (see Berndt~\cite[p. 223, 226]{Berndt-1}),
\begin{equation}\label{Ram-1}
\begin{split}
q \psi^4(q^2) &= \sum_{n=1}^{\infty} \fr{(2n-1)q^{2n-1}}{1-q^{4n-2}}  \\
\fr{\psi^3(q)}{\psi(q^3)} &= 1 + 3\sum_{n=1}^{\infty}\Big( \fr{q^{6n-5}}{1-q^{6n-5}}-\fr{q^{6n-1}}{1-q^{6n-1}} \Big).
\end{split}
\end{equation}
For other related identities involving Lambert series using we refer
the reader to Shen~\cite{Shen-1} and Liu~\cite{Liu-2003}.
It seems that not much is known about series representations of higher powers of
$q \psi^4(q^2)$ and $\fr{\psi^3(q)}{\psi(q^3)}$. We have the following contribution
concerning the squares of $q \psi^4(q^2)$ and $\fr{\psi^3(q)}{\psi(q^3)}$.
%
\begin{theorem}\label{main-thm-1}
We have
\[ {\mathrm (a)\quad}
\Big(\fr{\psi^3(q^3)}{\psi(q^9)}\Big)^2
= \left(1 + 3\sum_{n=1}^{\infty}\Big(\fr{q^{18n-15}}{1-q^{18n-15}}-\fr{q^{18n-3}}{1-q^{18n-3}} \Big) \right)^2
\]
\[
=
1 + 3\sum_{n=1}^{\infty}\Big(\fr{q^{6n}}{(1-q^{6n})^2} + \fr{2 q^{6n-3}}{(1-q^{6n-3})^2} -
\fr{9 q^{18n}}{(1-q^{18n})^2} \Big).
\]
\[ {\mathrm (b)\quad}
\big( q \psi^4(q^2) \big)^2
= \Big(\sum_{n=1}^{\infty} \fr{(2n-1)q^{2n-1}}{1-q^{4n-2}} \Big)^2
=\sum_{n=1}^{\infty}\Big(\fr{6 q^{8n-4}}{(1-q^{4n-2})^4}+\fr{q^{4n-2}}{(1-q^{4n-2})^2} \Big).
\]
\end{theorem}
The paper is organised as follows. As
the main tool in this note is the $q$-trigonometry of Gosper~\cite{Gosper},
the needed ingredients about this will be presented in Section~\ref{Facts-q-trig}.
In Section~\ref{sec:Gosper-sums} we will prove Theorem~\ref{Gosper-sums},
in Section~\ref{sec:Beyond-Gosper} we shall prove Theorem~\ref{Beyond-Gosper},
Section~\ref{sec:main-thm-1} is devoted to the proof of Theorem~\ref{main-thm-1},
and in Section~\ref{sec:main-integrals} we will prove Theorem~\ref{main-integrals}.
\section{Facts on Gosper's $q$-trigonometry}\label{Facts-q-trig}
In this section we collect the properties of Gosper's functions $\sin_q z$ and $\cos_q z$,
as defined in (\ref{sine-cos-q-prod}), which we need
to derive our results.
Just as the classical functions $\sin z$ and $\cos z$, it is easy to see that
 $\cos_q (z) = \sin_q \Big(\fr{\pi}{2} - z\Big)$,
the function $\sin_q z$ is odd and the function $\cos_q z$ is even and therefore
we have for any nonnegative integer $m$
\begin{equation}\label{zero-derivative}
\sin_q ^{(2m)} 0 = \cos_q ^{(2m+1)} 0 = \sin_q^{(2m+1)} \fr{\pi}{2} = \cos_q ^{(2m)} \fr{\pi}{2} = 0
\end{equation}
where derivatives are with respect to $z$.
Gosper
gave several $q$-trigonometric identities based on a computer algebra facility called MACSYMA.
For our current purposes, we shall need the following formulas which are marked with the same labels as in
Gosper~\cite{Gosper}.
\[   \label{q-Double} \tag{$q$-Double}
\sin_q(2z) = \frac{1}{2}\frac{\Pi_q}{\Pi_{q^4}} \sqrt{(\sin_{q^4}z)^2- (\sin_{q^2}z)^4 },
\]
\[
\label{q-Double-2} \tag{$q$-Double$_2$}
\sin_q(2z) = \fr{\Pi_q}{\Pi_{q^2}} \sin_{q^2} z \cos_{q^2} z,
\]
\[
\label{q-Double-3} \tag{$q$-Double$_3$}
\cos_q(2z) = (\cos_{q^2} z)^2 - (\sin_{q^2} z)^2,
\]
\[   \label{q-Double-4} \tag{$q$-Double$_4$}
\cos_q(2z) = \frac{1}{2}\frac{\Pi_q}{\Pi_{q^4}} \sqrt{(\cos_{q^4}z)^2- (\cos_{q^2}z)^4 },
\]
\[   \label{q-Double-5} \tag{$q$-Double$_5$}
\cos_q(2z) = (\cos_{q}z)^4- (\sin_{q}z)^4,
\]
\[ \label{q-Triple} \tag{$q$-Triple}
\sin_q(3z) = \frac{1}{3}\frac{\Pi_q}{\Pi_{q^9}} \sin_{q^9}z - \left(1+\frac{1}{3}\frac{\Pi_q}{\Pi_{q^9}}\right) (\sin_{q^3}z)^3,
\]
and
\[ \label{q-Triple-2} \tag{$q$-Triple$_2$}
\sin_q(3z) = \frac{\Pi_q}{\Pi_{q^3}} (\cos_{q^3} z)^2 \sin_{q^3}z -  (\sin_{q^3}z)^3.
\]
Gosper by combining (\ref{q-Double}) and (\ref{q-Double-2}) obtained
\begin{equation}\label{Pi-124}
\fr{\Pi_q^2}{\Pi_{q^2}\Pi_{q^4}}-\fr{\Pi_{q^2}^2}{\Pi_{q^4}^2} = 4.
\end{equation}
In addition to the previous formulas which are in one variable $z$,
Gosper~\cite[p. 101]{Gosper} also stated many multi-variable identities. For instance he conjectured that
\[ \label{q-Add-3} \tag{$q$-Add$_3$}
\sin_{q^3} x \sin_q (2y-x) - \sin_{q^3} y \sin_q (2x-y) = \cos_{q^3} y \cos_q (2x-y) -
\cos_{q^3} y \cos_q (2x-y).
\]
A proof for (\ref{q-Double-2}) can be found in~\cite{Bachraoui-1, Mezo}, a proof for
(\ref{q-Double-3}) is given in~\cite{Bachraoui-1}, proofs for (\ref{q-Triple}), (\ref{q-Triple-2}), and (\ref{q-Add-3}) are obtained in~\cite{Bachraoui-2}, and a proof for (\ref{q-Double}) is found
in~\cite{Bachraoui-3} . In addition,
it was observed in Gosper~\cite[p. 89-93]{Gosper} that (\ref{q-Double-4}) is equivalent to (\ref{q-Double})
and that (\ref{q-Double-5}) can be deduced by a combination of (\ref{q-Double}), (\ref{q-Double-3}), and
(\ref{q-Double-4}).
Furthermore, we shall need some explicit derivates with respect to the variable $z$ of the functions $\sin_q\pi z$ and $\cos_q\pi z$. To evaluate these derivatives,
we make an appeal to series expansions for $\sin_q\pi z$ and $\cos_q\pi z$ in powers of $(q^{-z}-q^z)$ which were stated without proof by Gosper~\cite[p. 98-99]{Gosper}.
We first confirm these series expansions as in the following theorem.
\begin{theorem}\label{sine-cosine-expand}
If $z$ is a complex number such that $|z|<1$, then
\begin{equation*}
\begin{split}
(a)\quad \sin_q \pi z
&= \Pi_q q^{z^2} \Big( (q^{-z}- q^{z}) -
(q^{-z}- q^{z})^3 \sum_{n=1}^{\infty} \frac{q^{2n}}{(1-q^{2n})^2} \\
& \quad + (q^{-z}- q^{z})^5
\frac{\left(\sum_{n=1}^{\infty} \frac{q^{2n}}{(1-q^{2n})^2}\right)^2
- \sum_{n=1}^{\infty} \frac{q^{4n}}{(1-q^{2n})^4}}{2} - \cdots \Big). \\
(b)\quad \cos_q \pi z
&= q^{z^2} \Big( 1 -
(q^{-z}- q^{z})^2 \sum_{n=1}^{\infty} \frac{q^{2n-1}}{(1-q^{2n-1})^2} \\
& \quad + (q^{-z}- q^{z})^4
\frac{\left(\sum_{n=1}^{\infty} \frac{q^{2n-1}}{(1-q^{2n-1})^2}\right)^2
- \sum_{n=1}^{\infty} \frac{q^{4n-2}}{(1-q^{2n-1})^4}}{2} - \cdots \Big).
\end{split}
\end{equation*}
\end{theorem}
\begin{proof}
Clearly,
\[
\begin{split}
\sin_q \pi z
&= q^{\left(z-\fr{1}{2}\right)^2} \prod_{n=1}^{\infty} \fr{(1-q^{2n-2z})(1-q^{2n+2z-2})}{(1-q^{2n-1})^2} \\
&= \fr{q^{\fr{1}{4}}}{(q;q^2)_{\infty}^2} q^{z^2-z}(1-q^{2z}) \prod_{n=1}^{\infty} (1-q^{2n-2z})(1-q^{2n+2z}) \\
&= \Pi_q q^{z^2}  (q^{-z}-q^z) \prod_{n=1}^{\infty} \left(1- (q^{-z}-q^z)^2 \fr{q^{2n}}{(1-q^{2n})^2} \right).
\end{split}
\]
Writing
\[
\prod_{n=1}^{\infty} \left(1- (q^{-z}-q^z)^2 \fr{q^{2n}}{(1-q^{2n})^2} \right) =
\sum_{n=0}^{\infty} a_n (q^{-z}-q^z)^n,
\]
we easily see that
\[
a_{2n+1}= 0,\quad a_0 = 1, \quad
a_2 = \sum_{n=1}^{\infty}\fr{q^{2n}}{(1-q^{2n})^2},\ \text{and\ }
a_4 = \sum_{k<l} \fr{q^{2k} q^{2l}}{(1-q^{2k})^2 (1-q^{2l})^2}.
\]
But
\[
\begin{split}
\sum_{k<l} \fr{q^{2k} q^{2l}}{(1-q^{2k})^2 (1-q^{2l})^2}
&=
\fr{1}{2} \left( \fr{q^2}{(1-q^{2})^2} + \fr{q^4}{(1-q^{4})^2} + \fr{q^6}{(1-q^{6})^2}+ \cdots   \right)^2 \\
&\quad
- \fr{1}{2} \left( \fr{q^4}{(1-q^{2})^4} + \fr{q^8}{(1-q^{4})^4} + \fr{q^{12}}{(1-q^{6})^4} + \cdots   \right) \\
&=
\fr{1}{2}\left( \sum_{n=1}^{\infty} \fr{q^{2n}}{(1-q^{2n})^2} \right)^2
- \fr{1}{2} \sum_{n=1}^{\infty} \fr{q^{4n}}{(1-q^{2n})^4}.
\end{split}
\]
This proves part (a). Part (b) follows similarly.
\end{proof}
From these expansions one can straightforwardly deduce the following explicit derivatives with
respect to $z$.
\begin{corollary}\label{derivatives-cor}
We have
\begin{align*}
\sin_q' 0
&= -\frac{2\ln q}{\pi}\Pi_q  \nonumber \\
\sin_q^{(3)} 0
&= -\frac{2(\ln q)^2}{\pi^3} \Pi_q
\left(6 +\ln q - 24\ln q \sum_{n=1}^{\infty}\frac{q^{2n}}{(1-q^{2n})^2} \right)  \nonumber \\
\cos_q'' 0
&= \frac{2\ln q}{\pi^2}
\left(1 - 4\ln q \sum_{n=1}^{\infty}\frac{q^{2n-1}}{(1-q^{2n-1})^2} \right)  \nonumber\\
\cos_q^{(4)} 0
&= \frac{4 (\ln q)^2}{\pi^4} \left( 3 - 8 (\ln q)^2 \sum_{n=1}^{\infty}\frac{q^{2n-1}}{(1-q^{2n-1})^2}
- 24 \ln q \sum_{n=1}^{\infty}\frac{q^{2n-1}}{(1-q^{2n-1})^2} \right. \nonumber \\
& \left. \quad + 48 (\ln q)^2 \Bigg( \Big( \sum_{n=1}^{\infty}\frac{q^{2n-1}}{(1-q^{2n-1})^2} \Big)^2
- \sum_{n=1}^{\infty}\frac{q^{4n-2}}{(1-q^{2n-1})^4} \Big) \vphantom{\int_1^2} \right) \nonumber \\
\end{align*}
\end{corollary}
\noindent
We note that the derivatives $\sin_q' 0$  and $\cos_q'' 0$  are also given in Gosper~\cite[p. 101]{Gosper}.
\section{Proof of Theorem \ref{Gosper-sums}}\label{sec:Gosper-sums}

\noindent
(a)\ Observe first the elementary fact that
\[
\sigma^{\ast}_{1(2)}(n) - 2 \sigma^{\ast}_{2(4)}(n)
= \begin{cases}
\sigma^{\ast}_{1(2)}(n) = \sigma (n) & \text{if $n$ is odd} \\
0 & \text{if $n$ is even},
\end{cases}
\]
where $\sigma^{\ast}_{b(a)}(n)$ is defined as in Remark~\ref{rmk:divisor}.
It follows that
\[
\sum_{n=1}^{\infty} \big(\sigma^{\ast}_{1(2)}(n) - 2 \sigma^{\ast}_{2(4)}(n) \big) q^n
= \sum_{n=0}^{\infty} \sigma(2n+1) q^{2n+1}.
\]
On the other hand,
by an identity of Ramanujan, see Berndt~\cite[p. 139]{Berndt-1},
\[
q\psi^4(q^2) = \sum_{n=0}^{\infty}\sigma(2n+1) q^{2n+1}.
\]
Now combine the above with the generating function~(\ref{divisor-generating}) to arrive at
\begin{equation}\label{easy-Gosper}
\sum_{n=1}^{\infty}\frac{q^{2n-1}}{(1-q^{2n-1})^2} - 2 \sum_{n=1}^{\infty}\frac{q^{4n-2}}{(1-q^{4n-2})^2} =
q \psi^4(q^2) = \Pi_{q^2}^2.
\end{equation}
Thus with the help of (\ref{easy-Gosper}) the following three formulas are equivalent:
\[
\sum_{n=1}^{\infty}\frac{q^n}{(1-q^n)^2} - 2 \sum_{n=1}^{\infty}\frac{q^{2n}}{(1-q^{2n})^2} =
\frac{1}{24}\left( \frac{\Pi_q^4}{\Pi_{q^2}^2}-1 \right) + \frac{2}{3}\Pi_{q^2}^2
\]
\[
\sum_{n=1}^{\infty}\frac{q^{2n}}{(1-q^{2n})^2} +
\sum_{n=1}^{\infty}\frac{q^{2n-1}}{(1-q^{2n-1})^2}
- 2 \sum_{n=1}^{\infty}\frac{q^{4n}}{(1-q^{4n})^2} - 2 \sum_{n=1}^{\infty}\frac{q^{4n-2}}{(1-q^{4n-2})^2}
\]
\[ =
\frac{1}{24}\left( \frac{\Pi_q^4}{\Pi_{q^2}^2}-1 \right) + \frac{2}{3}\Pi_{q^2}^2
\]
\[
\frac{1}{24}\left( \frac{\Pi_{q^2}^4}{\Pi_{q^4}^2}-1 \right) + \frac{2}{3}\Pi_{q^4}^2 + \Pi_{q^2}^2
=
\frac{1}{24}\left( \frac{\Pi_q^4}{\Pi_{q^2}^2}-1 \right) + \frac{2}{3}\Pi_{q^2}^2.
\]
The foregoing identity after simplification and rearrangement means
\[
\Pi_q^4 \Pi_{q^4}^2 = \Pi_{q^2}^6 + 16 \Pi_{q^2}^2 \Pi_{q^4}^4 + 8 \Pi_{q^2}^4 \Pi_{q^4}^2
\]
or, equivalently
\[
\left( \Pi_q^2 \Pi_{q^4} \right)^2 = \left( \Pi_{q^2}^3 + 4 \Pi_{q^2} \Pi_{q^4}^2 \right)^2.
\]
But the previous identity holds true since by~(\ref{Pi-124}) we have
\[
\Pi_q^2 \Pi_{q^4}  = \Pi_{q^2}^3 + 4 \Pi_{q^2} \Pi_{q^4}^2.
\]
This completes the proof of part (a). We postpone the proof of (b) as we need parts~(d, e).

\noindent
As to part (c), square both sides of (\ref{q-Double}), differentiate three times with respect
to $z$, and then let $z=0$ to obtain,
\[
128 \sin_q'0 \sin_q^{(3)}0 =
\fr{\Pi_q^2}{4\Pi_{q^2}^2} \Big(8 \sin_{q^4}'0 \sin_{q^4}^{(3)}0 - 24 (\sin_{q^2}'0)^4 \Big)
\]
which by appealing to Corollary~\ref{derivatives-cor}, simplifying, and rearranging reduces to
\[
 \fr{\Pi_q^2 \Pi_{q^2}^4}{\Pi_{q^4}^2} -  \Pi_q^2
=
8 \Pi_q^2 \Big(\sum_{n=1}^{\infty}\fr{q^{2n}}{(1-q^{2n})^2}
- 4 \sum_{n=1}^{\infty}\fr{q^{8n}}{(1-q^{8n})^2} \Big),
\]
or equivalently,
\[
\fr{1}{8}\left(\fr{\Pi_{q^2}^4}{\Pi_{q^4}^2}-1\right) =
\sum_{n=1}^{\infty}\fr{q^{2n}}{(1-q^{2n})^2} - 4 \sum_{n=1}^{\infty}\fr{q^{8n}}{(1-q^{8n})^2},
\]
which is equivalent to the desired formula.

\noindent
(d)\ Differentiating (\ref{q-Triple}) three times with respect to $z$ yields
\[
\begin{split}
27 \sin_q^{(3)} 3z &= \frac{\Pi_q}{3 \Pi_{q^9}} \sin_{q^9}^{(3)} z -
3\left(1+\frac{\Pi_q}{3 \Pi_{q^9}}\right)
\Big(2 (\sin_{q^3}'z)^3 \\
& \qquad\qquad + 6 \sin_{q^3} z \sin_{q^3}'z \sin_{q^3}'' z +
(\sin_{q^2} z)^2 \sin_{q^3}^{(3)} z \Big),
\end{split}
\]
which at $z=0$ boils down to
\[
27 \sin_q^{(3)} 0 = \frac{\Pi_q}{3 \Pi_{q^9}} \sin_{q^9}^{(3)} 0 -
6 \left(1+\frac{\Pi_q}{3 \Pi_{q^9}}\right) (\sin_{q^3}' 0)^3.
\]
Now by an appeal to the relations in Corollary~\ref{derivatives-cor} the foregoing identity implies
\[
- \frac{54 (\ln q)^2 \Pi_q}{\pi^3}
\left(6 +\ln q -24 \ln q \sum_{n=1}^{\infty}\frac{q^{2n}}{(1- q^{2n})^2}\right)
 = 6 \Big(1+\frac{\Pi_q}{3 \Pi_{q^9}} \Big) \frac{216 (\ln q)^3 \Pi_{q^3}^3}{\pi^3}
 \]
 \[
-\frac{54 (\ln q)^2 \Pi_q}{\pi^3}
\Big( 6 + 9 \ln q - 216 \ln q \sum_{n=1}^{\infty}\frac{q^{18n}}{(1- q^{18n})^2} \Big),
\]
which after straightforward simplifications becomes
\[
8 + 24 \sum_{n=1}^{\infty}\frac{q^{2n}}{(1- q^{2n})^2} -
216 \sum_{n=1}^{\infty}\frac{q^{18n}}{(1- q^{18n})^2} =
24 \frac{\Pi_{q^3}^3}{\Pi_q} \Big(1 + \frac{\Pi_q}{\Pi_{q^9}} \Big).
\]
Then dividing both sides by $24$ we arrive at the desired formula.

\noindent
(e)\ Letting $y=2x$ in the formula (\ref{q-Add-3}) we deduce
\[
\sin_{q^3} x \sin_q (3x) = \cos_{q^3} (2x) - \cos_{q^3} x \cos_q (3x).
\]
Differentiating twice with respect to $x$ and letting $x=0$ yield
\[
3 \sin_{q^3}'0 \sin_q'0 + 3 \sin_{q^3}'0 \sin_q'0 = 3 \cos_{q^3}'' 0 - 9 \cos_q'' 0,
\]
which simplifies to
\[
2 \sin_q'0 \sin_{q^3}'0 = \cos_{q^3}'' 0 - 3 \cos_q'' 0.
\]
Then by virtue of Corollary~\ref{derivatives-cor} we obtain
\[
\begin{split}
24 \frac{(\ln q)^2}{\pi^2} \Pi_q \Pi_{q^3} &=
6 \frac{\ln q}{\pi^2} \big(1-12 \ln q \sum_{n=1}^{\infty}\frac{q^{6n-3}}{(1- q^{6n-3})^2} \big) \\
& - 6 \frac{\ln q}{\pi^2} \big(1-4 \ln q \sum_{n=1}^{\infty}\frac{q^{2n-1}}{(1- q^{2n-1})^2} \big),
\end{split}
\]
or equivalently,
\[
\Pi_q \Pi_{q^3} = \sum_{n=1}^{\infty}\frac{q^{2n-1}}{(1- q^{2n-1})^2} - 3
\sum_{n=1}^{\infty}\frac{q^{6n-3}}{(1- q^{6n-3})^2},
\]
as desired.

\noindent
We now prove part~(b). By an appeal to parts (d, e) and Theorem~\ref{Beyond-Gosper} with $q$ replaced by $q^3$
and elementary simplifications we get
\[
\begin{split}
\fr{(\Pi_{q^3}^2 + 3\Pi_{q^9}^2)^2}{12\Pi_{q^3} \Pi_{q^9}} - \fr{1}{12}
& =
\fr{1}{12}\fr{\Pi_{q^3}^3}{\Pi_{q^9}} + \fr{3}{4}\fr{\Pi_{q^9}^3}{\Pi_{q^3}}
+ \fr{1}{2} \Pi_{q^3} \Pi_{q^9} - \fr{1}{12} \\
&=
\fr{1}{12} \left( 3 \sum_{n=1}^{\infty}\fr{q^{2n}}{(1-q^{2n})^2}
- 27 \sum_{n=1}^{\infty}\fr{q^{18n}}{(1-q^{18n})^2} \right) \\
& -\fr{1}{4}\left(\sum_{n=1}^{\infty}\fr{q^{2n}}{(1-q^{2n})^2}
- \sum_{n=1}^{\infty}\fr{q^{6n}}{(1-q^{6n})^2}
- \sum_{n=1}^{\infty}\fr{2 q^{6n-3}}{(1-q^{6n-3})^2} \right) \\
& + \fr{3}{4} \left(\sum_{n=1}^{\infty}\fr{q^{6n}}{(1-q^{6n})^2}
- \sum_{n=1}^{\infty}\fr{q^{18n}}{(1-q^{18n})^2}
- \sum_{n=1}^{\infty}\fr{2 q^{18n-9}}{(1-q^{18n-9})^2} \right) \\
&+ \fr{1}{2} \left(\sum_{n=1}^{\infty}\fr{q^{6n-3}}{(1-q^{6n-3})^2}
- \sum_{n=1}^{\infty}\fr{3 q^{18n-9}}{(1-q^{18n-9})^2} \right) \\
&=
-3 \sum_{n=1}^{\infty} \left(\fr{q^{18n}}{(1-q^{18n})^2}
+ \fr{q^{18n-9}}{(1-q^{18n-9})^2} \right) \\
&+ \sum_{n=1}^{\infty}\left(\fr{q^{6n}}{(1-q^{6n})^2}
+ \fr{q^{6n-3}}{(1-q^{6n-3})^2} \right) \\
&=
\sum_{n=1}^{\infty}\fr{q^{3n}}{(1-q^{3n})^2}
- 3 \sum_{n=1}^{\infty}\fr{q^{9n}}{(1-q^{9n})^2},
\end{split}
\]
which proves part (b).

\noindent
(f)\ Taking the fourth derivative at $z=0$ in (\ref{q-Double-5}) we find
\[
\cos_q^{(4)} 0 = 3 (\cos_q '' 0)^2 + 2 (\sin_q'0)^4.
\]
Then using the derivatives in Corollary~\ref{derivatives-cor} we deduce that
\begin{align*}
\frac{12 (\ln q)^2}{\pi^4} &- \frac{32 (\ln q)^4}{\pi^4} \sum_{n=1}^{\infty}\frac{q^{2n-1}}{(1- q^{2n-1})^2} -
\frac{96 (\ln q)^3}{\pi^4}\sum_{n=1}^{\infty}\frac{q^{2n-1}}{(1- q^{2n-1})^2} \\
&+ \frac{192 (\ln q)^4}{\pi^4} \Bigg( \Big(\sum_{n=1}^{\infty}\frac{q^{2n-1}}{(1- q^{2n-1})^2} \Big)^2
- \sum_{n=1}^{\infty}\frac{q^{4n-2}}{(1- q^{2n-1})^4} \Bigg) \\
=
\frac{12 (\ln q)^2}{\pi^4}& \Bigg( 1+16 (\ln q)^2
\Big(\sum_{n=1}^{\infty}\frac{q^{2n-1}}{(1- q^{2n-1})^2} \Big)^2 -
8 \ln q \sum_{n=1}^{\infty}\frac{q^{2n-1}}{(1- q^{2n-1})^2} \Bigg) \\
&- \frac{32 (\ln q)^4}{\pi^4} \Pi_q ^4,
\end{align*}
or, equivalently
\[
\frac{32 (\ln q)^4}{\pi^4} \Bigg( \sum_{n=1}^{\infty}\frac{q^{2n-1}}{(1- q^{2n-1})^2} +
\sum_{n=1}^{\infty}\frac{q^{4n-2}}{(1- q^{2n-1})^4} \Bigg)
= \frac{32 (\ln q)^4}{\pi^4} \Pi_q^4,
\]
which is clearly equivalent to the desired identity.
\section{Proof of Theorem~\ref{Beyond-Gosper}}\label{sec:Beyond-Gosper}
Differentiating (\ref{q-Triple-2}) three times with respect to $z$ , taking $z=0$, and simplifying yield
\[
27 \sin_q^{(3)} 0 = \fr{\Pi_q}{\Pi_{q^3}} (6\cos_{q^3}'' 0 \sin_{q^3}'' 0 + \sin_{q^3}^{(3)} )
- 6 (\sin_{q^3}' 0)^3.
\]
Now make an appeal to Corollary~\ref{derivatives-cor} and simplify to derive
\[
18\cdot 72 \Pi_q \sum_{n=1}^{\infty}\fr{q^{2n}}{(1-q^{2n})^2}
= 36\cdot 72 \Pi_q \Big( \sum_{n=1}^{\infty}\fr{q^{6n-3}}{(1-q^{6n-3})^2}
 + 18\cdot 72 \sum_{n=1}^{\infty}\fr{q^{6n}}{(1-q^{6n})^2} \Big) + 6^4 \Pi_{q^3}^3.
\]
The foregoing identity reduces to
\[
\sum_{n=1}^{\infty}\Big(\fr{q^{2n}}{(1-q^{2n})^2} - \fr{q^{6n}}{(1-q^{6n})^2}
-\fr{2q^{6n-3}}{(1-q^{6n-3})^2} = \fr{\Pi_{q^3}^3}{\Pi_q},
\]
which completes the proof. 
\section{Proof of Theorem~\ref{main-thm-1}}\label{sec:main-thm-1}
(a)\ 
On the one hand, by Theorem~\ref{Gosper-sums}(d) and Theorem~\ref{Beyond-Gosper} we get
\[
\begin{split}
\fr{\Pi_{q^3}^3}{\Pi_{q^9}}
&= 1 + 3 \sum_{n=1}^{\infty}\Big(\fr{q^{2n}}{(1-q^{2n})^2} - \fr{9 q^{18n}}{(1-q^{18n})^2} \Big)
 - 3\fr{\Pi_{q^3}^3}{\Pi_{q}} \\
&= 1+ 3 \sum_{n=1}^{\infty}\Big(\fr{q^{6n}}{(1-q^{6n})^2} + \fr{2q^{6n-3}}{(1-q^{6n-3})^2}
- \fr{9 q^{18n}}{(1-q^{18n})^2} \Big).
\end{split}
\]
On the other hand, by virtue of (\ref{Ram-1}) with $q$ replaced by $q^3$ we obtain
\[
\fr{\Pi_{q^3}^3}{\Pi_{q^9}}  = \Big(\fr{\psi^3(q^3)}{\psi(q^9)} \Big)^2
= \Big( 1 + 3 \sum_{n=1}^{\infty}
\big(\fr{q^{18n-15}}{1-q^{18n-15}} - \fr{q^{18n-3}}{1-q^{18n-3}} \big) \Big)^2.
\]
Now compare the above to complete the proof.

\noindent
(b)\ 
By Theorem~\ref{Gosper-sums}(f) with $q$ replaced by $q^2$ we derive
\[
\big( q \psi^4(q^2)\big)^2= \Pi_{q^2}^4 = 6 \sum_{n=1}^{\infty}\fr{q^{8n-4}}{(1-q^{4n-2})^4}
+ \sum_{n=1}^{\infty}\fr{q^{4n-2}}{(1-q^{4n-2})^2}.
\]
Now combine the previous formula with (\ref{easy-Gosper}) to conclude the proof.
\section{Proof of Theorem~\ref{main-integrals}}\label{sec:main-integrals}
\noindent
(a)\ From the fact $\Pi_q = q^{\fr{1}{4}} \psi^2(q)$ and Theorem~\ref{Gosper-sums}(f) we have
\[
\psi^8 (q) = 6 \sum_{n=1}^{\infty}\fr{q^{4n-3}}{(1-q^{2n-1})^4} + \sum_{n=1}^{\infty}\fr{q^{2n-2}}{(1-q^{2n-1})^2}.
\]
Moreover, we have the following elementary identities
\[
\int \fr{q^{2n-2}}{(1-q^{2n-1})^2}\, dq = \fr{1}{(2n-1)(1-q^{2n-1})}
\]
and
\[
\int \fr{q^{4n-3}}{(1-q^{2n-1})^4}\, dq = \fr{q^{2n-1}}{3(2n-1)(1-q^{2n-1})^3}-\fr{1}{6(2n-1)(1-q^{2n-1})^2}.
\]
It follows from the above that
\[
\begin{split}
\int \psi^8 (q) \,
&= \sum_{n=1}^{\infty}\fr{1}{2n-1}\Big(\fr{2 q^{2n-1}}{(1-q^{2n-1})^3}-\fr{1}{(1-q^{2n-1})^2}
 +\fr{1}{(1-q^{2n-1})}\Big) \\
&= \sum_{n=1}^{\infty}\fr{1}{2n-1}\fr{2q^{2n-1}-(1-q^{2n-1})+(1-q^{2n-1})^2}{(1-q^{2n-1})^3} \\
&= \sum_{n=1}^{\infty}\fr{q^{2n-1}+q^{4n-2}}{(2n-1)(1-q^{2n-1})^3},
\end{split}
\]
as desired.

\noindent
(b)\ By Theorem~\ref{Gosper-sums}(c) and the definition of $\Pi_q$ we have
 \[
 \fr{1}{8}\fr{\psi^8(q)}{\psi^4(q^2)} - \fr{1}{8}
 =\sum_{n=1}^{\infty}\fr{q^n}{(1-q^n)^2} - 4\sum_{n=1}^{\infty}\fr{q^{4n}}{(1-q^{4n})^2},
 \]
implying that
\[
\fr{1}{8}\fr{\psi^8(q)}{q \psi^4(q^2)} - \fr{1}{8q}
= \fr{d}{dq}\left(\sum_{n=1}^{\infty}\fr{1}{n}\Big(\fr{1}{1-q^n}-\fr{1}{1-q^{4n}} \Big) \right).
\]
Thus after simplification we deduce
\[
\int \fr{\psi^8(q)}{q \psi^4(q^2)} \, dq
= C + \log q + 8\sum_{n=1}^{\infty}\fr{q^n+q^{2n}+q^{3n}}{n(1-q^{4n})}
\]
which proves the first identity of part (b). As to the second formula, by Theorem~\ref{Gosper-sums}(a) we obtain
\[
\fr{1}{24}\fr{\psi^8(q)}{\psi^4(q^2)} - \fr{1}{24} + \fr{2}{3}q \psi^4(q^2)
= \sum_{n=1}^{\infty} \fr{q^n}{(1-q^n)^2} - 2\sum_{n=1}^{\infty}\fr{q^{2n}}{(1-q^{2n})^2}
\]
which with the help of~(\ref{equiv-int-1}) implies that
\[
\fr{1}{24}\fr{\psi^8(q)}{q \psi^4(q^2)} - \fr{1}{24q}
= \fr{d}{dq}\left(\sum_{n=1}^{\infty}\fr{1}{n}\Big(\fr{1}{1-q^n}-\fr{1}{1-q^{2n}} \Big)\right) +\fr{1}{24 q}
\]
\[
- \fr{2}{3} \fr{d}{dq}\left(\sum_{n=1}^{\infty}\fr{1}{2n-1}\fr{q^{2n-1}}{1-q^{4n-2}}\right).
\]
Then
\[
\int \fr{\psi^8(q)}{q \psi^4(q^2)} \, dq
= C + \log q + \sum_{n=1}^{\infty}\left(\fr{24 q^n}{n(1-q^{2n})}-\fr{16 q^{2n-1}}{(2n-1)(1-q^{4n-2})}\right)
\]
as desired.

\noindent
(c)\ From Theorem~\ref{Gosper-sums}(c), we deduce that
\[
\begin{split}
\sum_{n=1}^{\infty}\fr{q^{2n-1}}{(1-q^{2n-1})^2} - 4 \sum_{n=1}^{\infty}\fr{q^{8n-4}}{(1-q^{8n-4})^2}
&= \fr{1}{8}\Big( \fr{\Pi_q^4}{\Pi_{q^2}^2} - \fr{\Pi_{q^2}^4}{\Pi_{q^4}^2} \Big) \\
&= \fr{1}{8} \fr{\psi^8(q)}{\psi^4(q^2)} - \fr{1}{8} \fr{\psi^8(q^2)}{\psi^4(q^4)}.
\end{split}
\]
It follows that
\[
\fr{d}{dq}\sum_{n=1}^{\infty}\fr{1}{2n-1}\Big(\fr{1}{1-q^{2n-1}}- \fr{1}{1-q^{8n-4}}\Big)
= \fr{1}{8} \fr{\psi^8(q)}{q\psi^4(q^2)} - \fr{1}{8} \fr{\psi^8(q^2)}{q\psi^4(q^4)},
\]
that is,
\[
\int \fr{\psi^8(q^2)}{q \psi^4(q^4)} \, dq
= \int \fr{\psi^8(q)}{q \psi^4(q^2)} \, dq - 8\sum_{n=1}^{\infty}\fr{1}{2n-1}\Big(\fr{1}{1-q^{2n-1}}- \fr{1}{1-q^{8n-4}}\Big).
\]
Now use part~(b) and simplify to arrive at the result.

\noindent
(d)\ By an appeal to Theorem~\ref{Beyond-Gosper} we deduce that
\[
q \fr{\psi^6(q^3)}{\psi^2(q)}
= \fr{d}{dq}\sum_{n=1}^{\infty}\Big(\fr{1}{2n(1-q^{2n})}-\fr{1}{6n(1-q^{6n})}+\fr{2}{(6n-3)(1-q^{6n-3})} \Big)
\]
and therefore
\[
\int \fr{q\psi^6(q^3)}{\psi^6(q)} \, dq
= C + \sum_{n=1}^{\infty}\Big(\fr{2+3q^{2n}+3q^{4n}}{6n(1-q^{6n})}-\fr{2}{(6n-3)(1-q^{6n-3})} \Big),
\]
which yields the desired formula.

\noindent
(e)\ By virtue of Theorem~\ref{Gosper-sums}(d) and easy simplification we deduce
\[
\fr{d}{dq}\sum_{n=1}^{\infty} \fr{1}{2n}\Big(\fr{1}{1-q^{2n}}-\fr{1}{1-q^{18n}}\Big)
= q\fr{\psi^6(q^3)}{\psi^2(q)} + \fr{1}{3} \fr{\psi^6(q^3)}{q \psi^2(q^9)} - \fr{1}{3q}
\]
Now proceed as in the previous parts and use part~(d) to conclude the proof.

\noindent
(f)\ Using Theorem~\ref{Gosper-sums}(b), we get
\[
\sum_{n=1}^{\infty}\fr{q^n}{(1-q^n)^2}- 3\sum_{n=1}^{\infty}\fr{q^{3n}}{(1-q^{3n})^2}
= \fr{1}{12}\fr{\psi^6(q)}{\psi^2(q^3)} + \fr{3}{4}\fr{q^2 \psi^6(q^3)}{\psi^2(q)}
+ q \psi^2(q)\psi^2(q^3) - \fr{1}{12},
\]
implying that
\[
\fr{d}{dq}\sum_{n=1}^{\infty} \fr{1}{n}\Big(\fr{1}{1-q^{n}}-\fr{1}{1-q^{3n}}\Big)
= \fr{1}{12}\fr{\psi^6(q)}{q\psi^2(q^3)} + \fr{3}{4}\fr{q \psi^6(q^3)}{\psi^2(q)}
+  \psi^2(q)\psi^2(q^3) - \fr{1}{12q}.
\]
Now combine part~(d) with (\ref{equiv-int-2}) and simplify to establish the desired integral.
%

%
\end{document}